\documentclass[11pt]{amsart}
\usepackage{amssymb,amsmath,amsthm,amsfonts,amsopn,hyperref,url}
\usepackage{enumitem}
\theoremstyle{plain}
\newtheorem{thm}{Theorem}[section]

\newtheorem{prop}[thm]{Proposition}
\newtheorem{lemma}[thm]{Lemma}
\newtheorem{cor}[thm]{Corollary}

\theoremstyle{definition}
\newtheorem{defn}[thm]{Definition}
\newtheorem*{defn*}{Definition}
\newtheorem{example}[thm]{Example}
\newtheorem*{example*}{Example}

\theoremstyle{remark}
\newtheorem{rmk}[thm]{Remark}
\newtheorem*{rmk*}{Remark}

\newcommand{\field}[1]{\mathbb{#1}}

\newcommand{\N}{\field{N}}

\newcommand{\R}{\field{R}}
\newcommand{\Z}{\field{Z}}
\newcommand{\ideal}[1]{\mathfrak{#1}}
\newcommand{\m}{\ideal{m}}

\DeclareMathOperator{\Max}{Max}
\DeclareMathOperator{\Spec}{Spec}

\DeclareMathOperator{\orc}{\Omega}

\renewcommand{\phi}{\varphi}

\author{Neil Epstein}
\address{Department of Mathematical Sciences \\ George Mason University \\ Fairfax, VA  22030}
\email{nepstei2@gmu.edu}

\author{Jay Shapiro}
\address{Department of Mathematical Sciences \\ George Mason University \\ Fairfax, VA  22030}
\email{jshapiro@gmu.edu}

\title{The Ohm-Rush content function}


\date{December 27, 2014}
\begin{document}
\begin{abstract}
The content of a polynomial over a ring $R$ is a well understood notion.  Ohm and Rush generalized this concept of a content map to an arbitrary ring extension of $R$, although it can behave quite badly.
We examine five properties an algebra may have with respect to this function -- content algebra, weak content algebra, semicontent algebra (our own definition), Gaussian algebra, and Ohm-Rush algebra.  We show that the Gaussian, weak content, and semicontent algebra properties are all transitive.  However, transitivity is unknown for the content algebra property.  We then compare the Ohm-Rush notion with the more usual notion of content in the power series context.
We show that many of the given properties coincide for the power series extension map over a valuation ring of finite dimension, and that they are equivalent to the value group being order-isomorphic to the integers or the reals.  Along the way, we give a new characterization of Pr\"ufer domains.
\end{abstract}

\maketitle

\section{Introduction}
Let $R$ be a commutative ring.  The polynomial algebra extension $R \rightarrow R[X]$ enjoys many remarkable properties.  Beyond the fact that it is faithfully flat and that prime ideals extend to prime ideals, there is the \emph{content} function, which sends any element of $R[X]$ to a characteristic ideal of $R$ -- namely, the ideal generated by the coefficients of the polynomial in question.

Of course, the polynomial algebra extension is not the only ring extension that enjoys such properties (e.g. certain semigroup rings \cite{No-content} and power series extensions \cite{nmeSh-DMpower}).  Accordingly, Ohm and Rush \cite{OhmRu-content} (see also Eakin and Silver \cite{EakSi-almost}) defined the notion of a \emph{content algebra} $R \rightarrow S$ (see Definition~\ref{def:ca}) as a faithfully flat extension with a so-called \emph{content function} (see Definition~\ref{def:orc}) from $S$ to the finitely generated ideals of $R$ satisfying certain properties enjoyed by the content function of polynomial extensions.  This idea was used and expanded upon in Mott and Schexnayder \cite{MoSc-semivalue}, Rush \cite{Ru-content}, and more recently by Nasehpour \cite{Nas-zdcontent, Nas-ABconj}.

Later, Rush \cite{Ru-content} defined the notion of a \emph{weak content algebra} (see Definition~\ref{def:wca}), a condition that is easier to check than the defining condition of content algebras.  This condition was explored further, in a geometric context, by G. Picavet \cite{Pic-contenu}.  It is clear that any content algebra is a faithfully flat weak content algebra, which raises the question of whether the converse is true.  Rush claimed not \cite{Ru-content}; he showed that if $R$ is Noetherian, then the power series extension $R \rightarrow R[\![X]\!]$ is a weak content algebra, and then he claimed that there exist Noetherian $R$ for which this is not a content algebra.  However, we showed in \cite{nmeSh-DMpower} that indeed, whenever $R$ is Noetherian, $R \rightarrow R[\![X]\!]$ is a content algebra.

In this paper, we develop further properties of the Ohm-Rush content function and both of the above kinds of algebras, as well as a variant by Nasehpour (the \emph{Gaussian algebra}, which is stronger than either of Rush's) and our own variant (the \emph{semicontent algebra}, which sits between Rush's two).  When the base ring is Noetherian, we characterize semicontent algebras in terms of primary ideals (see Proposition~\ref{pr:primaries}).  Also, we show that the weak content, semicontent, and Gaussian algebra properties are \emph{transitive} (see Corollary~\ref{cor:weaktrans} and Theorems~\ref{thm:transitive} and \ref{thm:Gausstransitive}), which is unknown for the content algebra property.

Finally, we explore the situation of power series extensions.  In that case, there is already a `common sense' notion of the content of a power series (see Definition~\ref{def:pscontent}), which sometimes coincides with the Ohm-Rush content and sometimes does not (see \S\ref{sec:power} for precise statements).  We show among other things that when $V$ is a finite-dimensional valuation ring, the extension $V \rightarrow V[\![X]\!]$ is Gaussian iff it is weak content iff the value group is isomorphic to $\mathbb R$ or $\mathbb Z$ (see Corollary~\ref{cor:1dimV}).  We also (see Theorem~\ref{thm:Pr}) give a new characterization of Pr\"ufer domains.

\section{Basics}

We start with some background on general notions of ``content'', introduced by Ohm and Rush in 1972 \cite{OhmRu-content}:

\begin{defn}\label{def:orc}
Let $R$ be a ring, $M$ an $R$-module, and $f\in M$.  Then the \emph{(Ohm-Rush) content} of $f$ is given by  \[
\orc(f) := \bigcap\{I \subseteq R \text{ ideal } \mid f \in IM\}.
\]
If $f\in \orc(f)M$ for all $f\in M$, we say that $M$ is an \emph{Ohm-Rush module}; if $M$ is moreover an $R$-algebra, we say that it is an \emph{Ohm-Rush algebra} over $R$.\footnote{This is not the terminology of Ohm and Rush, but we feel that the new terminology will create less confusion.  Also, our reason for using a symbol other than the `$c$' that they use will become evident in \S\ref{sec:power}.  The symbol $\orc$ stands for `Ohm'.}
\end{defn}

The following is a useful alternate characterization:

\begin{lemma}\label{lem:fgdef}
Let $M$ be any $R$-module and $f \in M$.  Then \[
\orc(f) = \bigcap\{I \subseteq R \mid f \in IM \text{ and } I \text{ is finitely generated}\}.
\]
\end{lemma}

\begin{proof}
Evidently the displayed intersection contains $\orc(f)$.  Conversely, let $a \in R$ be a member of the given intersection. Let $J$ be an ideal of $R$ such that $f\in JM$.  Then there exist $j_1, \dotsc, j_n \in J$ and $m_1, \dotsc, m_n \in M$ such that $f=\sum_{i=1}^n j_i m_i$.  But then setting $I := (j_1, \dotsc, j_n)$, $I$ is a finitely generated ideal and $f \in IM$, whence $a \in I \subseteq J$.  Thus, $a\in \orc(f)$.
\end{proof}

\begin{defn}\label{def:ca}
If $R \rightarrow S$ is a faithfully flat Ohm-Rush algebra and satisfies the additional ``Dedekind-Mertens''
condition that for all $f, g \in S$, there is some $n \in \N$ such that $\orc(f)^n \orc(g) = \orc(f)^{n-1} \orc(fg)$, then one says \cite{OhmRu-content} that $S$ is a \emph{content algebra} over $R$.
\end{defn}

Weakening this latter concept somewhat, we make the following definition:
\begin{defn}\label{def:semi}
A faithfully flat Ohm-Rush $R$-algebra $S$ is a \emph{semicontent algebra} if for any multiplicative set $W \subseteq R$, whenever $f, g \in S$ such that $\orc(f)_W = R_W$, we have $\orc(fg)_W = \orc(g)_W$.
\end{defn}

The final preliminary definition we encounter is that of a \emph{weak content algebra}, introduced by Rush \cite{Ru-content} in 1978:
\begin{defn}\label{def:wca}
An Ohm-Rush algebra $R \rightarrow S$ is called a \emph{weak content algebra} if for any $f, g\in S$, one has $\orc(f)\orc(g) \subseteq \sqrt{\orc(fg)}$.
\end{defn}
Note that it follows from being an Ohm-Rush algebra that $\orc(fg) \subseteq \orc(f)\orc(g)$ and $\orc(f+g) \subseteq \orc(f)+\orc(g)$.  Also, recall \cite[Theorem 1.2]{Ru-content} that if $R \rightarrow S$ is an Ohm-Rush algebra, then it is a weak content algebra if and only if for every $P \in \Spec R$, either $PS = S$ or $PS \in \Spec S$.

\begin{prop}
Any ring map $R \rightarrow S$ satisfies the following implications:
\begin{center}
content algebra $\implies$ semicontent algebra $\implies$ weak content algebra.
\end{center}
\end{prop}

\begin{proof}
For the first implication, since the property of being a content algebra localizes \cite[Theorem 6.2]{OhmRu-content}, it is enough to show the condition of Definition~\ref{def:semi} when $W=\{1\}$.  Let $f,g\in R$ with $\orc(f)=R$.  By definition of content algebra, there is some $n$ such that $\orc(g) = \orc(f)^n \orc(g) = \orc(f)^{n-1} \orc(fg) = \orc(fg)$.

For the second implication, let $P \in \Spec R$ and $f, g \in S$ with $fg \in PS$ but $f \notin PS$.  Then $\orc(f) \nsubseteq P$, so there is some $d\in \orc(f) \setminus P$.  Let $W := \{d^n \mid n \in \N\}.$  Then $\orc(f)_W = R_W$, whence $\orc(g)_W = \orc(fg)_W \subseteq PR_W$.  Then for any $x \in \orc(g)$, there is an $n \in \N$ such that $d^n x  \in P$.  But $d \notin P$, whence $x\in P$.  Thus, $\orc(g) \subseteq P$ so that $g \in PS$, as was to be shown.
\end{proof}

In Noetherian contexts, the following proposition is helpful:
\begin{prop}\label{pr:primaries}
Let $R$ be a ring, and let $S$ be a faithfully flat Ohm-Rush algebra.  If $R \rightarrow S$ is a semicontent algebra, then for any primary ideal $Q$ of $R$, $QS$ is a primary ideal of $S$.  The converse holds whenever $R$ is Noetherian.
\end{prop}

\begin{proof}
First, suppose that $R \rightarrow S$ is a semicontent algebra.  Let $Q$ be a primary ideal of $R$, and let $fg \in QS$ with $f \notin \sqrt{QS} = \sqrt{Q}S$ (where the latter equality follows from the fact that $R \rightarrow S$ is a weak content algebra).  Then $\orc(f) \nsubseteq \sqrt{Q}$.  Let $d\in \orc(f) \setminus \sqrt{Q}$, and set $W := \{d^n \mid n \in \N\}$.  Then $\orc(f)_W = R_W$, so by the semicontent condition, $\orc(g)_W = \orc(fg)_W \subseteq QR_W$, whence (since $\orc(g)$ is a finitely generated ideal of $R$) there is some $n$ with $d^n \orc(g) \subseteq Q$.  But $d \notin \sqrt{Q}$.  As $Q$ is a primary ideal, it then follows that $\orc(g) \subseteq Q$.  Thus, $g \in QS$, as was to be shown.

Conversely, suppose that $R$ is Noetherian and that primary ideals extend to primary ideals.  As this is a property that survives localization at any multiplicative subset of $R$, we may take $f, g\in S$ and assume that $\orc(f) = R$.  It will be enough to show that $\orc(fg) = \orc(g)$.  Since $R$ is Noetherian, we may take a primary decomposition of $\orc(fg)$, namely: \[
\orc(fg) = Q_1 \cap \cdots \cap Q_k,
\]
where each $Q_i$ is  primary.  For each $i$, we have $fg \in Q_i S$, but $f \notin \sqrt{Q_i S}$ (since $\orc(f) = R$).  Therefore, since $Q_i S$ is primary, it follows that $g \in Q_i S$, whence $\orc(g) \subseteq Q_i$.  Since this holds for all $1\leq i \leq k$, it follows that $\orc(g) \subseteq \orc(fg)$, as was to be shown.
\end{proof}

\section{Transitivity}\label{sec:trans}
Given any property a ring extension may have, it is interesting to know whether and when the property  is transitive.  In this section, we recall that the Ohm-Rush property is transitive, we show that the Ohm-Rush content function itself is transitive in this context, and we show that the weak content, semicontent, and Gaussian algebra properties are all transitive.

\begin{defn}
As we deal with multiple ring extensions at once, it will be essential, given a ring homomorphism $\phi: S \rightarrow T$, to describe the (Ohm-Rush) content of $f\in T$ \emph{relative} to $S$, in which case we will write  $\orc_{TS}(f)$.  If there is no chance for confusion we may suppress the subscript.
\end{defn}

If $W$ is a multiplicatively closed subset of the ring $R$, and $a\in R$, then we use $a/1$ to denote the image of $a$ in $R_W$.

\begin{lemma}\label{lem:formula} Let $S$ be a faithfully flat Ohm-Rush $R$-algebra. Let $W$ be a multiplicatively closed subset of $R$.  Then
for all $f\in S$ and $w\in W$ we have
  $$\orc_{S_WR_W}(f/w) = \orc_{SR}(f)R_W.$$
Moreover $S_W$ is  faithfully flat  Ohm-Rush $R_W$-algebra.
\end{lemma}
\begin{proof} The formula and the fact that $S_W$ is an Ohm-Rush $R_W$-algebra is \cite[Theorem 3.1]{OhmRu-content}.   From the formula and the fact that $S$ is faithfully flat over $R$ we deduce that for all $x \in R_W$ and $y \in S_W$, $\orc_{S_WR_W}(xy) = x\orc_{S_WR_W}(y)$ and $\orc_{S_WR_W}(S_W)= R_W$.  Hence $S_W$ is faithfully flat over $R_W$.
\end{proof}

\begin{prop} \label{pr:equiv} Let $S$ be a faithfully flat Ohm-Rush $R$-algebra.  Then $S$ is a semicontent $R$-algebra if and only if for each each $g\in S$, $P\in \Spec(R)$ and  $f\in S\setminus PS$, there exists $t\in R\setminus P$ such that $t\orc(g)\subseteq \orc(fg)$.
\end{prop}
\begin{proof}
First assume that $S$ is a semicontent $R$-algebra and let $P\in \Spec(R)$, $g\in S$, and $f\in S\setminus PS$. Let $W=R\setminus P$.  Then $S_W$ is an Ohm-Rush $R_W$-algebra and  $\orc_{S_WR_W}((f/1)(g/1)) = \orc_{SR}(fg)R_W$ by Lemma~\ref{lem:formula}.  Moreover, as $f \in S\setminus PS$, $f/1$ is not in the extension of the unique maximal ideal of $R_W$.  Hence it follows that  $\orc_{S_WR_W}(f/1)= R_W$.  Thus  $\orc_{SR}(fg)R_W = \orc_{S_WR_W}((f/1)(g/1)) = \orc_{S_WR_W}(g/1) = \orc_{SR}(g)R_W$.  Since the latter ideal is finitely generated, there exists $x\in W$ such that $x\orc_{SR}(g)\subseteq \orc_{SR}(fg)$.

For the converse, let $W$ be a multiplicatively closed subset of $R$.  We must show that if  $f\in S_W$ and $g\in S_W$ with $\orc_{S_WR_W}(f) = R_W$, then $\orc_{S_WR_W}(fg) = \orc_{S_WR_W}(g)$.  We already know that $\orc_{S_WR_W}(fg) \subseteq \orc_{S_WR_W}(g)$.  It suffices to show that for any $Q\in \Spec(R_W)$, $\orc_{S_WR_W}(g)_Q \subseteq \orc_{S_WR_W}(fg)_Q $. Write $f=f'/w_1$ and $g=g'/w_2$ with $f',g' \in R$ and $w_1,w_2 \in W$.  Also let $P\in \Spec(R)$ such that $PR_W = Q$.  Clearly $f'\not\in PS$, thus by hypothesis there exists $t\in R\setminus P$ such that $t\orc_{SR}(g') \subseteq \orc_{SR}(f'g')$.   It follows that $t/1 \in R_W\setminus Q$ and $(t/1)\orc_{S_WR_W}(g) \subseteq \orc_{S_WR_W}(fg)$, and so locally the two ideals are equal.  Thus they are equal ideals of $R_W$.
\end{proof}

\begin{cor} \label{cor:local}
Let $S$ be a semicontent $R$-algebra and let $W\subset R$ be a multiplicatively closed set.   Then $ S_W$ is a semicontent $R_W$-algebra.
\end{cor}

\begin{proof}
By Lemma~\ref{lem:formula}, $S_W$ is a faithfully flat Ohm-Rush $R_W$-algebra.  Since every prime ideal of $R_W$ is the extension of a prime ideal of $R$, it follows from Proposition~\ref{pr:equiv} that $S_W$ is a semicontent $R_W$-algebra.
\end{proof}

\begin{rmk} \rm It is now clear that if $S$ is a faithfully flat Ohm-Rush $R$-algebra, then $S$ is a semicontent $R$-algebra if and only if for all $P\in \Spec(R)$, and for all $g\in S_W$ ($W=R\setminus P$) and $f\in S_W$ where $\orc_{S_WR_P}(f) =R_P$, we have $\orc_{S_WR_P}(fg) = \orc_{S_WR_P}(g) $.
\end{rmk}

\begin{prop} \label{pr:morelocal} Let $S$ be  a semicontent $R$-algebra.  Let $V$ be a multiplicatively closed subset of $S$ such that $\orc(a) =R$ for all $a\in V$. Then
 $S_V$ is a semicontent $R$-algebra.
\end{prop}
\begin{proof} We first claim that for $f\in S$ and $t\in V$, $\orc_{S_VR}(f/t) =\orc_{SR}(f)$.  As $S$ is an Ohm-Rush $R$-algebra, we have $f \in \orc_{SR}(f)S$, so that $f/1 \in \orc_{SR}(f)S_V$.  Thus $f/t \in \orc_{SR}(f)S_V$, and so $\orc_{S_VR}(f/t)\subseteq \orc_{SR}(f)$.

  Conversely let $I$ be an ideal of $R$ such that $f/t\in IS_V = (IS)S_V$.  Then $t'f \in IS$ for some $t'\in V$. Therefore $\orc_{SR}(t'f) \subseteq I$.  But $\orc_{SR}(t') =R$ and since $S$ is a semicontent $R$-algebra, we have $\orc_{SR}(t'f) = \orc_{SR}(f)$.  Thus the latter ideal is contained in $I$, and so $\orc_{SR}(f) \subseteq \orc_{S_VR}(f/t)$, which proves the claim.

  This claim then proves that $S_V$ is an Ohm-Rush $R$-algebra as it is now clear that $f\in \orc_{S_VR}(f/t)$ for all $f\in S$ and $t\in V$, whence $f/t \in \orc_{S_VR}(f/t)$.

  We next claim that $S_V$ is faithfully flat over $R$.   The flatness is clear as $S$ is flat over $R$.   Hence it suffices to show that if $P$ is a prime ideal of $R$, then $PS_V\neq S_V$.   However,   $PS_V = S_V$ would imply that there exists some $t\in PS\cap V$.  This in turn would imply that $\orc_{SR}(t)\subseteq P$, a contradiction to the assumption that $\orc(a) =R$ for all $a\in V$.

  To finish the proof, we apply Proposition~\ref{pr:equiv}.  To that end, let $P$ be a prime ideal of $R$, $g\in S_V$ and $f\in S_V\setminus PS_V$.  Write $f=f'/u$ and $g = g'/v$,  where $f', g'\in S$ and $u,v \in V$.   Note that $f'\in S\setminus PS$.  Thus there exists $t\in R\setminus P$ such that $t\orc_{SR}(g') \subseteq \orc_{SR}(f'g')$.  Combining this with the above  claim we have \[
   t\orc_{S_VR}(g) = t\orc_{SR}(g') \subseteq \orc_{SR}(f'g') =\orc_{S_VR}(fg)
 \]
 which is exactly what we need to finish the proof.
\end{proof}

The following result is analogous to \cite[Theorem 6.2]{OhmRu-content}.

\begin{cor} Let $S$ be a semicontent $R$-algebra, let $V$ be a multiplicatively closed subset of $S$ and let $W =V\cap R$.  Suppose that for every $v\in V$, $\orc(v)\cap W \neq \emptyset$.  Then $S_V$ is a semicontent $R_W$ algebra and for $f\in S$ and $v\in V$ we have
 $$\orc_{S_VR_W}(f/v) = \orc_{SR}(f)R_W.$$
\end{cor}
\begin{proof} By Corollary~\ref{cor:local}, $S_W$ is a semicontent $R_W$-algebra.  Let $V'$ denote the image of $V$ in $S_W$.  It follows that each $v'\in V'$ satisfies $\orc_{S_WR_W}(v') = R_W$.  Hence by Proposition~\ref{pr:morelocal}, $S_V = (S_W)_{V'}$ is a semicontent $R_W$-algebra.  Finally the formula holds by successive applications of the formulas in Corollary~\ref{cor:local} and in (the proof of) Proposition~\ref{pr:morelocal}.
\end{proof}

In the sequel, we will need to know what the content of an \emph{ideal} is.  Accordingly, if $R \rightarrow S$ is a ring map and $J \subseteq S$ is an ideal, we set \[
\orc(J) := \bigcap \{\text{ideals } I \subseteq R \mid J \subseteq IS\}
\]
When $R$ is an \emph{Ohm-Rush} $R$-algebra, we also have $\orc(J) = \sum_{h \in J} \orc(h)$, and hence $J \subseteq \orc(J)S$ for all ideals $J$ of $R$.

\begin{lemma}
\label{lem:moreformula}
Let $R \subseteq S \subseteq T$ be  rings and homomorphisms such that $S$ is Ohm-Rush over $R$ and $T$ is Ohm-Rush over $S$.   Then $T$ is Ohm-Rush over $R$ and  $\orc_{SR}(\orc_{TS}(f)) =\orc_{TR}(f)$ for all $f\in T$.
\end{lemma}

\begin{proof} By \cite[1.2(ii)]{OhmRu-content}, $T$ is an Ohm-Rush $R$-algebra, so all we need to do is show the equality. We first show that for $f\in T$, $\orc_{SR}(\orc_{TS}(f))\subseteq \orc_{TR}(f)$.  For any $f\in T$, we have $f \in \orc_{TR}(f)T = [\orc_{TR}(f)S]T$, with the inclusion following since $T$ is Ohm-Rush over $R$.  But then by definition of $\orc_{TS}$, we then have that $\orc_{TS}(f) \subseteq \orc_{TR}(f)S$.  Finally, by the definition of $\orc_{SR}(-)$ applied to ideals, we have $\orc_{SR}(\orc_{TS}(f)) \subseteq \orc_{TR}(f)$.

To show that $\orc_{SR}(\orc_{TS}(f)) \supseteq \orc_{TR}(f)$, it suffices to show that $f \in \orc_{SR}(\orc_{TS}(f))T$.  By applying the comment preceding the lemma to the ideal $\orc_{TS}(f)$ of $S$, we have $\orc_{TS}(f)\subseteq \orc_{SR}(\orc_{TS}(f))S$ since $S$ is an Ohm-Rush $R$-algebra.  Hence, $f \in \orc_{TS}(f)T \subseteq (\orc_{SR}(\orc_{TS}(f))S)T = \orc_{SR}(\orc_{TS}(f))T$.
\end{proof}

\begin{cor}\label{cor:weaktrans}
Let $R \rightarrow S \rightarrow T$ be rings and homomorphisms such that $S$ is a weak content $R$-algebra and $T$ is a weak content $S$-algebra.  Then $T$ is a weak content $R$-algebra.
\end{cor}

\begin{proof}
Let $P \in \Spec R$.  Then either $PS = S$ or $PS \in \Spec S$, since $S$ is weak content over $R$.  In the former case, we have $PT = (PS)T =ST= T$, while in the latter we again have either $(PS)T = T$ or $(PS)T \in \Spec T$ since $T$ is weak content over $S$.  As $T$ is also an Ohm-Rush $R$-algebra, we are done.
\end{proof}

We next show the transitivity of the semicontent property.

\begin{thm} \label{thm:transitive} Let $T$ be a semicontent $S$-algebra and $S$ a semicontent $R$-algebra.  Then $T$ is a semicontent $R$-algebra.
\end{thm}
\begin{proof} As faithful flatness is transitive, it follows from Lemma~\ref{lem:moreformula} that $T$ is a faithfully flat Ohm-Rush $R$-algebra.  Now we apply Proposition~\ref{pr:equiv} again.  Accordingly, let $P\in \Spec(R)$, $g\in T$ and $f\in T\setminus PT$.  We must find $x\in R\setminus P$ such that $x\orc_{TR}(g) \subseteq \orc_{TR}(fg)$.
By Corollary~\ref{cor:local}  we may assume that $R$ is local with maximal ideal $P$ and that $\orc_{TR}(f) = R$.  We will then show that $\orc_{TR}(g) = \orc_{TR}(fg)$.   Note that the latter ideal is always contained in the former.

 Observe that $PS$ is a prime ideal of $S$, since $S$ is a faithfully flat weak content $R$-algebra.
 Since $T$ is a semicontent $S$-algebra and $g\not\in (PS)T =PT$, there exists $u\in S\setminus PS$ such that $u\orc_{TS}(g) \subseteq \orc_{TS}(fg)$.  Thus \[
    \orc_{SR}(u\orc_{TS}(g))\subseteq \orc_{SR}(\orc_{TS}(fg)). \]
   Moreover, since $u\not\in PS$, where $P$ is the unique maximal ideal of $R$, it follows that $\orc_{SR}(u) =R$.   Thus $ \orc_{SR}(u\orc_{TS}(g)) = \orc_{SR}(\orc_{TS}(g))$.  Therefore
  \[ \orc_{TR}(g) = \orc_{SR}(\orc_{TS}(g)) \subseteq \orc_{SR}(\orc_{TS}(fg)) = \orc_{TR}(fg), \]
   where the two equalities are by Lemma~\ref{lem:moreformula}.   This is sufficient to make the two ideals on the ends equal, giving us the result.
 \end{proof}

\begin{example} \rm Let $R$ be a Noetherian ring and let $X$ and $Y$ be indeterminates.  Let $S= R(X)$, the localization of $R[X]$ at the set of polynomials with content 1, and let $T =S[\![Y]\!]$.   Then $T$ is content $S$-algebra \cite{nmeSh-DMpower} and $S$ is a content $R$-algebra
\cite[Example 6.3(a)]{OhmRu-content}.  It follows from  Theorem~\ref{thm:transitive} that $T$ is a semicontent $R$-algebra.  It is unknown as to whether or not $T$ is a content $R$-algebra.
\end{example}

Finally, we consider the following related property due to Peyman Nasehpour \cite{Nas-ABconj}.

\begin{defn}\label{def:Gauss}
An Ohm-Rush algebra $R \rightarrow S$ is \emph{Gaussian} if for all $f,g \in S$, we have $\orc(fg) = \orc(f)\orc(g)$.
\end{defn}

We now show that the Gaussian property is also transitive.

\begin{thm}\label{thm:Gausstransitive}
Let $S$ be a Gaussian $R$-algebra and let $T$ be a Gaussian $S$-algebra.  Then $T$ is a Gaussian $R$-algebra.
\end{thm}

\begin{proof}
Let $f, g \in T$.  We want to show that $\orc_{TR}(f) \orc_{TR}(g) \subseteq \orc_{TR}(fg)$, since the other containment is automatic.  So pick any $a \in \orc_{TS}(f)$ and $b \in \orc_{TS}(g)$.  Then \[
\orc_{TR}(f) = \orc_{SR}(\orc_{TS}(f)) = \sum_{a \in \orc_{TS}(f)} \orc_{SR}(a),
\]
with the first equality by Lemma~\ref{lem:moreformula} and the second equality by definition of the Ohm-Rush content of an ideal.  The analogous formula holds for $\orc_{TR}(g)$.  Hence, \begin{align*}
\orc_{TR}(f) \orc_{TR}(g) &= \left(\sum_{a \in \orc_{TS}(f)} \orc_{SR}(a)\right) \cdot \left(\sum_{b \in \orc_{TS}(g)} \orc_{SR}(b)\right) \\
&= \sum_{a \in \orc_{TS}(f)} \sum_{b \in \orc_{TS}(g)} \orc_{SR}(a) \orc_{SR}(b) \\
&=  \sum_{a \in \orc_{TS}(f)} \sum_{b \in \orc_{TS}(g)} \orc_{SR}(a\cdot b) \\
&\subseteq \sum_{a, b \text{ as above}} \orc_{SR}\left(\orc_{TS}(f) \cdot \orc_{TS}(g)\right) = \orc_{SR}(\orc_{TS}(fg)) \\
&= \orc_{TR}(fg),
\end{align*}
where the third equality holds because $S$ is Gaussian over $R$, the penultimate equality holds because $T$ is Gaussian over $S$, and the final equality holds by Lemma~\ref{lem:moreformula}, thus completing the proof.
\end{proof}

\section{Power series extensions}\label{sec:power}
We consider the case $S=R[\![X]\!]$, where $X$ is a single variable.  We showed in \cite{nmeSh-DMpower} that if $R$ is Noetherian, then a Dedekind-Mertens type lemma holds for this situation.  We now investigate how our result fits into the Ohm-Rush context.
\begin{defn}\label{def:pscontent}
If $R$ is a commutative ring and $f = \sum_i a_i X^i \in R[\![X]\!]$, then $c(f)$ is the ideal of $R$ generated by the elements $a_i$ -- the `common-sense' content of the power series.
\end{defn}

\begin{lemma}\label{lem:orcc}
Let $R$ be a commutative ring, $S=R[\![X]\!]$, and $f\in S$.  If $J$ is a finitely generated ideal, then $f\in JR[\![X]\!]$ iff $c(f) \subseteq J$.  In fact, $\orc(f)$ is the intersection of all the finitely generated ideals that contain $c(f)$.  Hence, $c(f) \subseteq \orc(f)$, with equality if $R$ is Noetherian.
\end{lemma}

\begin{proof}
For the first statement: The `only if' direction is clear.  As for the `if' direction, suppose $c(f) \subseteq J$.  Let $J = (d_1, \dotsc, d_k)$ and let $f = \sum_n a_n X^n$.  Then each $a_n\in J$, so there exist $b_{in} \in R$, $1\leq i \leq k$, such that $a_n = \sum_{i=1}^k b_{in} d_i$, whence \[
f = \sum_n a_n X^n = \sum_n (\sum_{i=1}^k b_{in}d_i) X^n = \sum_{i=1}^k d_i (\sum_n b_{in} X^n) \in JR[\![X]\!].
\]

Now since by the previous lemma $\orc(f)$ is the intersection of those finitely generated ideals $J$ such that $f\in JR[\![X]\!]$, we now know that it is the intersection of those finitely generated ideals $J$ that contain $c(f)$, and the second statement follows.
\end{proof}

\begin{thm}
Let $R$ be a Noetherian ring.  Then $R[\![X]\!]$ is a content $R$-algebra.
\end{thm}

\begin{proof}
This is essentially a restatement of the main theorem of \cite{nmeSh-DMpower}.  Indeed, that theorem says that for any $f, g \in R[\![X]\!]$ there is some $n$ such that $c(f)^n c(g) = c(f)^{n-1}c(fg)$.  But by Lemma~\ref{lem:orcc}, $\orc = c$ in this situation.  Moreover, $R[\![X]\!]$ is faithfully flat over $R$ because $R$ is coherent.  As for the Ohm-Rush property, let $f = \sum_{k=0}^\infty f_k X^k$ be a power series ($f_k \in R$) and let $a_1, \dotsc, a_n$ be a finite generating set for $c(f)$.  Then each $f_k = \sum_{i=1}^n r_{ik} a_i$ for some $r_{ik} \in R$, and then letting $g_i := \sum_{k=0}^\infty r_{ik} X^i$, we have $f = \sum_{i=1}^n a_i g_i \in (a_1, \dotsc, a_n)R[\![X]\!] = \orc(f) R[\![X]\!]$.
\end{proof}

\begin{lemma}\label{lem:ctable}
Let $R$ be a commutative ring.  Then $R[\![X]\!]$ is an Ohm-Rush $R$-algebra if and only if for any countably generated ideal $I$ of $R$, there is a unique smallest finitely generated ideal $J$ of $R$ such that $I \subseteq J$.  In this case, for any $f\in R[\![X]\!]$, if $c(f)=I$, then $\orc(f)=J$.
\end{lemma}

\begin{proof}
First suppose that $R[\![X]\!]$ is an Ohm-Rush $R$-algebra.  Let $I = (a_n \mid n \in \N_0)$ be a countably generated ideal of $R$, and set $f := \sum_n a_n X^n \in R[\![X]\!]$.  Then $c(f) = I$, so $\orc(f)$ is the intersection of all finitely generated ideals that contain $I$.  But by the Ohm-Rush property, $\orc(f)$ is itself finitely generated, so it must be the smallest such ideal that contains $I$.

Conversely, suppose the given condition holds.  Let $f =\sum_{n=0}^\infty a_n X^n \in R[\![X]\!]$, let $I := (a_n \mid n \in \N_0)$, and let $J$ be the unique smallest finitely generated ideal such that $I \subseteq J$.  Then clearly $\orc(f) = J$, and by the previous lemma, $f \in JR[\![X]\!]$, whence $R[\![X]\!]$ is an Ohm-Rush $R$-algebra
\end{proof}

\begin{prop}\label{pr:corc}
Suppose $R[\![X]\!]$ is an Ohm-Rush $R$-algebra, and let $f \in R[\![X]\!]$.  Then $c(f)$ is finitely generated if and only if $\orc(f) = c(f)$.
\end{prop}

\begin{proof}
The ``only if" direction follows from the previous lemma.  The ``if'' direction holds because (Ohm-Rush) content ideals are always finitely generated in the context of an Ohm-Rush module.
\end{proof}

\begin{cor}\label{cor:corcN}
The functions $c, \orc: R[\![X]\!] \rightarrow  \{$ideals of $R\}$ coincide if and only if $R$ is Noetherian.
\end{cor}

\begin{proof}
The set of ideals $\{c(f) \mid f \in R[\![X]\!]\}$ consists of all the finitely and countably generated ideals of $R$.  So by Proposition~\ref{pr:corc}, the functions $c$ and $\orc$ will coincide if and only if every countably generated ideal is finitely generated.  But we show below that this latter condition is the same as Noetherianness.

To see this, assume that every countably generated ideal is finitely generated, and suppose by way of contradiction that $R$ is not Noetherian.  Then there is a strictly ascending chain of ideals $(0)=I_0 \subsetneq I_1 \subsetneq I_2 \subsetneq \cdots$.  For each $j\geq 1$, let $a_j \in I_j \setminus I_{j-1}$, and set $H := (a_1, a_2, a_3, \ldots)$.  Then $H$ is countably generated, hence finitely generated, so $H=(a_1, \dotsc, a_n)$ for some $n$.  But then $a_{n+1} \in H =(a_1, \dotsc, a_n) \subseteq I_n$, yielding the desired contradiction.
\end{proof}

We now turn to power series over valuation rings.  First a fairly general fact, whose proof follows the outline of the traditional development of Gauss's lemma.  For this, recall that an integral domain $R$ is called \emph{Pr\"ufer} if $R_\m$ is a valuation ring for all maximal ideals $\m$ of $R$.

\begin{prop}\label{pr:Pruf}
Let $R$ be a Pr\"ufer domain (e.g. any Dedekind domain), and let $S$ be a faithfully flat Ohm-Rush algebra over $R$.  Suppose that for all $\m \in \Max R$, $\m S \in \Spec S$.  Then $S$ is a Gaussian $R$-algebra.
\end{prop}

\begin{proof}
First suppose that $(R,\m)$ is local (i.e. a valuation domain).  Let $f, g$ be nonzero elements of $S$.  Since $\orc(f)$ and $\orc(g)$ are finitely generated ideals of $R$, they must be principal; say $\orc(f) = (a)$ and $\orc(g)=(b)$.  Then $f \in aS$ and $g \in bS$; say $f=af'$ and $g=bg'$.  Recall \cite[Corollary 1.6]{OhmRu-content} that since $R \rightarrow S$ is a flat Ohm-Rush algebra, we have $\orc(rh)= r\orc(h)$ for any $r\in R$ and $h\in S$.  Thus $\orc(fg) = \orc(af'bg') = ab\orc(f'g')$, and $\orc(f)\orc(g) = \orc(af')\orc(bg') = a\orc(f')b\orc(g') = ab\orc(f')\orc(g')$.  Therefore since $ab$ is a nonzero element of a domain (and hence regular), we have $\orc(fg) = \orc(f)\orc(g) \iff \orc(f'g') = \orc(f')\orc(g')$.  But $aR = \orc(f) = a\orc(f')$ and $bR = \orc(g) = b\orc(g')$, so since both $a$, $b$ are nonzero, we have $\orc(f') = \orc(g') = R$.  Thus, we may replace $f$, $g$ by $f'$, $g'$ respectively and thereby assume that $\orc(f)=\orc(g)=R$.  If $\orc(fg) \neq R$, then $\orc(fg) \subseteq \m$.  But then $fg \in \m S$ by the Ohm-Rush property, and $\m S \in \Spec S$, so without loss of generality $f \in \m S$.  But then $\orc(f) \subseteq \m$, a contradiction.  Hence, $\orc(fg) = R = \orc(f)\orc(g)$.

In the general case, let $f, g$ be nonzero elements of $R$, and let $\m \in \Max R$.  By Lemma~\ref{lem:formula}, we have that $S_\m$ is a faithfully flat Ohm-Rush algebra over the valuation domain $R_\m$, and that $\orc_{S_\m R_\m}(f/1) = (\orc_{SR}(f))_\m$, and similarly for $g$.  Then from the local case, we have \begin{align*}
\orc_{SR}(fg)_\m &= \orc_{S_\m R_\m}(fg/1) = \orc((f/1)(g/1)) = \orc_{S_\m R_\m}(f/1) \orc_{S_\m R_\m}(g/1) \\
&= \orc_{SR}(f)_\m \orc_{SR}(g)_\m = (\orc_{SR}(f)\orc_{SR}(g))_\m.
\end{align*}
Hence, $\orc_{SR}(fg) = \orc_{SR}(f)\orc_{SR}(g)$.
\end{proof}

Combined with an old result, we get a new characterization of Pr\"ufer domains.
\begin{thm}\label{thm:Pr}
Let $R$ be an integral domain.  Then $R$ is Pr\"ufer iff for any faithfully flat Ohm-Rush $R$-algebra $S$ such that $\m S \in \Spec S$ for all $\m \in \Max R$, $S$ is a Gaussian $R$-algebra.
\end{thm}

\begin{proof}
The ``only if'' direction is Proposition~\ref{pr:Pruf}.  For the other direction, let $S=R[X]$.  Clearly the algebra $R \rightarrow S=R[X]$ satisfies the condition on maximal ideals as well as the Ohm-Rush condition, so by assumption, $S$ is Gaussian over $R$.  That is, for all $f, g \in R[X]$, we have $\orc(f)\orc(g) = \orc(fg)$, i.e. $c(f)c(g) = c(fg)$.  But Tsang \cite{Ts-Gauss} showed that any integral domain $R$ with the latter property must be Pr\"ufer.
\end{proof}

\begin{thm}\label{thm:glb}
Let $V$ be a valuation ring with value group $G$. Then $V[\![X]\!]$ is an Ohm-Rush $V$-algebra if and only if any countable subset of positive elements of $G$ has a greatest lower bound in $G$.
\end{thm}

\begin{proof}
First suppose $V[\![X]\!]$ is an Ohm-Rush $V$-algebra.  Let $\{g_n \mid n \in \N_0\}$ be a countable set of positive elements of $G$, let $a_n \in V$ with $v(a_n) = g_n$ for all $n$, and let $I = (a_n \mid n \in \N_0)$.  Then by Lemma~\ref{lem:ctable}, there is a unique smallest finitely generated ideal $J$ containing $I$.  But since $V$ is a valuation ring, $J=(a)$ for some $a\in V$.  Then $v(a) = \inf\{v(a_n) \mid n \in \N_0\}$.

Conversely, suppose the greatest lower bound property holds.  Let $I=(a_n \mid n \in \N_0)$ be a countably generated ideal of $V$.  Let $u= \inf\{v(a_n) \mid n\in \N_0\}$; this exists in $G$ by hypothesis.  Let $a\in V$ with $v(a)=u$.  Then $I \subseteq (a)$.  Moreover, suppose there is some finitely generated (i.e., principal) ideal $(b)$ such that $I \subseteq (b)$.  Then for each $n$, we have $a_n = bc_n$ for some $c_n \in V$.  Let $c\in V$ with $v(c) = \inf_n v(c_n)$.  Then $v(a) = v(bc)$, whence $a$ is a multiple of $b$, so that $(a)$ is the smallest finitely generated ideal that contains $I$.
\end{proof}

We note that the above condition on the group $G$ is equivalent to the statement that every countable set of principal ideals in $V$ has a least upper bound in the ordered set of principal ideals of $V$.

\begin{thm} \label{thm:gauss}
Let $V$ be a valuation ring of dimension at least $2$.  Then $V[\![X]\!]$ is not an Ohm-Rush $V$-algebra.
\end{thm}

\begin{proof} As noted above, it suffices to show that there exist a countable set of principal ideals that does not have a least upper bound (in the set of principal ideals of $V$).  Since $\dim V\geq 2$, it follows from \cite[Theorem 11]{Kap-CR} that there exists $Q, P \in \Spec V$, such that $Q\subset P$ are adjacent primes where $P$ is not the maximal ideal.

 For now we will assume that $P$ is a height one prime ideal of $V$ that is not the maximal ideal $\m$.  We break this into two cases.  In the first, we suppose that $V_P$ (which necessarily has dimension one) is a DVR.   Thus $PV_P$ is principal.   Let $a\in P$ be such that the element $\frac{a}{1}\in V_P$ generates $PV_P$.   Let $b\in \m\setminus P$ and consider the set of principal ideals $\{(\frac{a}{b^m})\}_{m\geq 0}$ in $V$.  Suppose that $(c)$ is an upper bound of this set (in the collection of principal ideals of $V$).   Then $(cb)$ is a smaller upper bound on this set, since for any $m\geq 0$, we have $\frac{a}{b^{m+1}} \in (c)$, whence $\frac{a}{b^m} \in (cb)$, and since $(cb) \subsetneq (c)$.  Hence there is no least upper bound.

Next suppose that $V_P$ is not discrete, in which case the value group of $V_P$ must be a dense subgroup of the reals \cite[Theorem 10.7]{Mats}. Hence there is a set of elements $\{a_i\}_{i\in \mathbb{N}} \in P$ such that the images of the $a_i$'s in $V_P$ generate $PV_P$.  Let $(c)$ be an upper bound for the set of ideals $\{(a_i)\}$.  If $c\in P$, then $PV_P = \bigcup (\frac{a_i}{1}) \subseteq (\frac{c}{1}) \subseteq PV_P$, which shows that $PV_P$ is a principal ideal -- a contradiction.   Hence $c\not\in P$.   If $c$ is a unit, then $(c)$ is not a least upper bound in the set of principal ideals, since for any $d \in \m \setminus P$, $(d)$ is an upper bound and $(d) \subsetneq R = (c)$.   Finally let $c\in \m\setminus P$, in which case $(c^2) \subset (c)$ is a smaller upper bound.   Thus, the  set of ideals $\{(a_i)\}$ does not have a least upper bound.

For the general case, let $P$ be an arbitrary nonmaximal prime ideal of $V$ that has an immediate predecessor $Q$ in $\Spec R$.  Pass to the ring $V/Q$.   Then by the above, there exists a countable set of principal ideals $\{(a_i)\}_{i>0}$ in $V/Q$ that does not have a least upper bound in the set of principal ideals of $V/Q$.  Let $b_i \in V$ have image $a_i$ in $V/Q$.  As $V$ is a valuation ring, $Q \subset (b_i)$ for all $i$.  It follows that the set $\{(b_i)\}$ does not have a least upper bound in the set of principal ideals of $V$, and so our proof is complete.
\end{proof}

\begin{cor}\label{cor:1dimV}
Let $V$ be a valuation ring  with nontrivial value group $G$.  Then the following are equivalent:\begin{enumerate}[label=(\alph*)]
\item $V[\![X]\!]$ is an Ohm-Rush $V$-algebra.
\item $V[\![X]\!]$ is a Gaussian $V$-algebra.
\item $G \cong \R$ or $G \cong \Z$.
\end{enumerate}
\end{cor}

\begin{proof}
(a) $\implies$ (c): Theorem \ref{thm:gauss} implies that  $\dim V =1$, in which case $G$ is order isomorphic to a subgroup of $\mathbb{R}$.  If $G$ is not isomorphic to $\R$ or $\Z$, then $G$ is a non-closed subgroup of $\R$, and again the condition on greatest lower bounds from Theorem~\ref{thm:glb} will fail.

(c) $\implies$ (a): This follows directly from Theorem~\ref{thm:glb} and the order properties of $\Z$ and $\R$.

(b) $\implies$ (a): By definition.

(a) $\&$ (c) $\implies$ (b): First note that $V[\![X]\!]$ is faithfully flat over $V$ because $V$ is coherent.  As for the rest, by Proposition~\ref{pr:Pruf}, it is enough to show that the maximal ideal of $V$ extends to a prime ideal of $V[\![X]\!]$.  But when $G\cong\Z$, this is standard, and when $G\cong\R$, this is by \cite[Lemma 1]{ArBr-val}.
\end{proof}

Finally, we see what happens over fields:

\begin{prop}\label{pr:field}
Let $k$ be a field, and $S$ a $k$-algebra.  Then \begin{enumerate}
\item $S$ is a faithfully flat Ohm-Rush $k$-algebra.
\item $S$ is a weak content algebra over $k$ iff it is Gaussian over $k$ iff it is an integral domain.
\end{enumerate}
\end{prop}

\begin{proof}
Faithful flatness is automatic.  As for the Ohm-Rush property, note that for any $s\in s$, $\orc(s) = k$ if $s\neq 0$.  Thus, $s \in S = \orc(s)S$ for all nonzero $s$, and $0 \in 0S = \orc(0)s$.

For the second part, if $S$ is not an integral domain, then there exist nonzero $f, g \in S$ such that $fg=0$, so that $c(f)c(g) = kk=k$, whereas $c(fg) = c(0) = (0)$, which do not have the same radical.  On the other hand, if $S$ is an integral domain, then for any nonzero pair of $f,g\in S$, we have $c(fg) = k = kk= c(f)c(g)$.
\end{proof}

\begin{example}
The above proposition provides examples of chains of ring maps $R \rightarrow S \rightarrow T$ such that $S$ and $T$ are Gaussian over $R$, but $T$ is neither faithfully flat nor Ohm-Rush over $S$.  For example, let $R$ be any field, let $S = R[X]$, and let $T =S_W$, where $W = \{X^n \mid n \in  \N\}$.  Then since $X^n T = T$ for all $n \in \N$, we have \[
\orc_{TS}(1_T) \subseteq \bigcap_{n \in \N} X^n S = (0_S),
\]
but $1_T \notin (0_S) T = 0$.  On the other  hand, $T$ and $S$ are faithfully flat and Gaussian over $R$ by Proposition~\ref{pr:field}.
\end{example}

\section{Questions}
\subsection{\ } Let $R \rightarrow S$ be a faithfully flat ring map.  We know that if $S$ is a content $R$-algebra, it is a semicontent $R$-algebra, and if it is a semicontent $R$-algebra, it is a weak content $R$-algebra.  But are either of these implications reversible?  That is, is there a distinction between these concepts among faithfully flat algebras?

\subsection{\ }
As shown in \S\ref{sec:trans}, the properties of being a weak content algebra, semicontent algebra, and Gaussian algebra are all transitive.  Also, the property of being an Ohm-Rush algebra is known to be transitive \cite[1.2(ii)]{OhmRu-content}.  But is the property of being a content algebra transitive?

\section*{Acknowledgments}
We are grateful to the referee for showing us how to remove unnecessary assumptions from Theorem~\ref{thm:gauss} and Corollary~\ref{cor:1dimV}.

\providecommand{\bysame}{\leavevmode\hbox to3em{\hrulefill}\thinspace}
\providecommand{\MR}{\relax\ifhmode\unskip\space\fi MR }
\providecommand{\MRhref}[2]{%
  \href{http://www.ams.org/mathscinet-getitem?mr=#1}{#2}
}
\providecommand{\href}[2]{#2}


\begin{thebibliography}{Kap70}

\bibitem[AB73]{ArBr-val}
Jimmy~T. Arnold and James~W. Brewer, \emph{When {$(D[[X]])_{P[[X]]}$} is a
  valuation ring}, Proc. Amer. Math. Soc. \textbf{37} (1973), no.~2, 326--332.

\bibitem[ES]{nmeSh-DMpower}
Neil Epstein and Jay Shapiro, \emph{A {D}edekind-{M}ertens theorem for power
  series rings}, arXiv:1402.1100 [math.AC], to appear in Proc. Amer. Math. Soc.

\bibitem[ES72]{EakSi-almost}
Paul Eakin and James Silver, \emph{Rings which are almost polynomial rings},
  Trans. Amer. Math. Soc. \textbf{174} (1972), 425--449.

\bibitem[Kap70]{Kap-CR}
Irving Kaplansky, \emph{Commutative rings}, Allyn and Bacon Inc., Boston, 1970.

\bibitem[Mat86]{Mats}
Hideyuki Matsumura, \emph{Commutative ring theory}, Cambridge Studies in
  Advanced Mathematics, no.~8, Cambridge Univ. Press, Cambridge, 1986,
  Translated from the {Japanese} by {M.} {Reid}.

\bibitem[MS76]{MoSc-semivalue}
Joe~L. Mott and Michel Schexnayder, \emph{Exact sequences of semi-value
  groups}, J. Reine Angew. Math. \textbf{283/284} (1976), 388--401.

\bibitem[Nas10]{Nas-zdcontent}
Peyman Nasehpour, \emph{Zero-divisors of content algebras}, Arch. Math. (Brno)
  \textbf{46} (2010), no.~4, 237--249.

\bibitem[Nas14]{Nas-ABconj}
\bysame, \emph{On the {A}nderson-{B}adawi $\omega_{R[X]}({I}[{X}]) =
  \omega_{R}({I})$ conjecture}, arXiv:1401.0459, 2014.

\bibitem[Nor59]{No-content}
Douglas~G. Northcott, \emph{A generalization of a theorem on the content of
  polynomials}, Proc. Cambridge Philos. Soc. \textbf{35} (1959), 282--288.

\bibitem[OR72]{OhmRu-content}
Jack Ohm and David~E. Rush, \emph{Content modules and algebras}, Math. Scand.
  \textbf{31} (1972), 49--68.

\bibitem[Pic85]{Pic-contenu}
Gabriel Picavet, \emph{Propri\'et\'es et applications de la notion de contenu},
  Comm. Algebra \textbf{13} (1985), no.~10, 2231--2265.

\bibitem[Rus78]{Ru-content}
David~E. Rush, \emph{Content algebras}, Canad. Math. Bull. \textbf{21} (1978),
  no.~3, 329--334.

\bibitem[Tsa65]{Ts-Gauss}
Hwa Tsang, \emph{Gauss's lemma}, Ph.D. thesis, University of Chicago, 1965.

\end{thebibliography}
\end{document}